\documentclass[reqno]{amsart}
\usepackage{amsmath, amssymb, amsthm, epsfig, enumerate}
\usepackage{hyperref, latexsym}
\usepackage{url}
\usepackage[mathscr]{euscript}

\usepackage{color}
\usepackage{fullpage} 
\usepackage{setspace}

\def\today{\ifcase\month\or
  January\or February\or March\or April\or May\or June\or
  July\or August\or September\or October\or November\or December\fi
  \space\number\day, \number\year}

\DeclareMathOperator{\supp}{\mathrm{supp}}

 \newtheorem{theorem}{Theorem}
  
 \newtheorem{lemma}[theorem]{Lemma}
 \newtheorem{proposition}[theorem]{Proposition}
 \newtheorem{corollary}[theorem]{Corollary}
 \theoremstyle{definition}
 \newtheorem{definition}[theorem]{Definition}
 \newtheorem{example}{Example}
 \theoremstyle{remark}
 \newtheorem{remark}[theorem]{Remark}

 \newcommand{\ft}{\widehat}
 \newcommand{\mc}{\mathcal}

 \newcommand{\D}{\mc{D}}
 
 \newcommand{\F}{\mc{F}}
 \newcommand{\GG}{\mathfrak{G}}
 \newcommand{\G}{\mc{G}}

 \newcommand{\Sw}{\mc{S}}

 \newcommand{\LL}{\mc{L}}
 \newcommand{\M}{\mc{M}}
\newcommand{\T}{\mathbb{T}}

 \newcommand{\C}{\mathbb{C}}
 \newcommand{\R}{\mathbb{R}}

 \newcommand{\Z}{\mathbb{Z}}

 \newcommand{\dt}{\text{\rm d}t}

 \newcommand{\dx}{\text{\rm d}x}

    \renewcommand{\d}{\text{\rm d}}

\newcommand{\lb}{\left\lbrace}
\newcommand{\rb}{\right\rbrace}
\newcommand{\ra}{\rightarrow}
\newcommand{\dint}{\displaystyle\int}
\newcommand{\dsum}{\displaystyle\sum}
\newcommand{\dprod}{\displaystyle\prod}

\newcommand{\hf}{\tfrac{1}{2}}

\begin{document}

\title[]{One-sided Band-Limited Approximations of Some Radial Functions}
\author[Gon\c{c}alves, Kelly and Madrid]{Felipe Gon\c{c}alves, Michael Kelly and Jose Madrid}
\date{\today}
\address{IMPA-I{\tiny{NSTITUTO}} {\tiny{DE}} M{\tiny{ATEMATICA}} P{\tiny{URA}} {\tiny{E}} A{\tiny{PLICADA}}, E{\tiny{STRADA}} D{\tiny{ONA}} C{\tiny{ASTORINA}}, 110, R{\tiny{IO}} {\tiny{DE}} J{\tiny{ANEIRO}}, 22460-320, B{\tiny{RAZIL}} }
\email{ffgoncalves@impa.br}

\address{D{\tiny{EPARTMENT}} {\tiny{OF}} M{\tiny{ATHEMATICS}}, U{\tiny{NIV.}} {\tiny{OF}} T{\tiny{EXAS}},1  U{\tiny{NIVERSITY}} S{\tiny{TA.}} C1200 A{\tiny{USTIN,}} T{\tiny{EXAS}} 78712 }
\email{mkelly@math.utexas.edu}

\address{IMPA-I{\tiny{NSTITUTO}} {\tiny{DE}} M{\tiny{ATEMATICA}} P{\tiny{URA}} {\tiny{E}} A{\tiny{PLICADA}}, E{\tiny{STRADA}} D{\tiny{ONA}} C{\tiny{ASTORINA}}, 110, R{\tiny{IO}} {\tiny{DE}} J{\tiny{ANEIRO}}, 22460-320, B{\tiny{RAZIL}} }
\email{josermp@impa.br}

\allowdisplaybreaks
\numberwithin{equation}{section}

\begin{abstract}
	We construct majorants and minorants of a Gaussian function in Euclidean space that have Fourier transforms supported in a box. The majorants that we construct are shown to be extremal and our minorants are shown to be {\it asymptotically extremal} as the sides of the box become uniformly large. We then adapt the Distribution and Gaussian Subordination methods of \cite{CLV} to the multidimensional setting to obtain majorants and minorants for a class of radial functions. Periodic analogues of the main results are proven and applications to Hilbert-type inequalities are given.
\end{abstract}

\maketitle
\setcounter{tocdepth}{1}
%\tableofcontents

%\begin{abstract}
%In this work, we obtain optimal majorants of exponential type with respect to parallelograms in $\R^d$ for a wide class of radial functions which have a certain Gaussian subordination. These majorants minimize the $L^1(\R^d,dx)$ distance to the target function. Also, we were able construct non--optimal minorants of exponential type with respect parallelograms in $\R^d$, that, under certain restrictions on the target function, are asymptotically optimal for large parallelograms. In fact, they converge exponentially fast to the optimal answer. Next, we use the developed theory to construct one--sided approximations by trigonometric polynomials. Finally, we give applications to multidimensional versions of Hilbert--Type inequalities.
%\end{abstract}

\section{Introduction}
In this paper we address a class of problems that have come to be known as {\it Beurling-Selberg extremal problems}. The most well-known example of such a problem is due to Selberg himself \cite{M1978, S2,V}. Given an interval $I\subset \R$ and $\delta>0$, Selberg constructed an integrable function $M(x)$ that satisfies 
\begin{enumerate}[(i)]
\item $\hat{M}(\xi)=0$ if $|\xi|>\delta$ where $\hat{M}(\xi)$ is the Fourier transform (see \S\ref{prelim}) of $M(x)$, 
\item $M(x)\geq \chi_{I}(x)$ for each $x\in\R$, and 
\item $M(x)$ has the smallest integral\footnote{This was shown by Selberg in the case when $\delta\mathrm{Length}(I)\in\Z$.If  $\delta\mathrm{Length}(I)\not\in\Z$, then the minimal integral was found by BF Logan (\cite{Logan}, unpublished) and Littmann \cite{Littmann2013}.} among all functions satisfying (i) and (ii). 
\end{enumerate}
The key constraint (which is common in Beurling-Selberg problems) is condition (ii), that $M(x)$ {\it majorizes} the characteristic function of $I$. Such problems are sometimes called {\it one-sided} approximation problems. 

Most work on these sorts of problems have been focused on solving Selberg's problem but with $\chi_{I}(x)$ replaced by a different single variable function, such as $e^{-\lambda|x|}$. Such problems are considered in \cite{CL,CL2,CG, CLV,CV2,CV3,L1, Littmann2006, Littmann2013, V}. Some work has also been done in the several variables setting. Shortly after his construction of his function $M(x)$, Selberg was able to construct majorants and minorants of the characteristic function of a box whose Fourier transforms are supported in a (possibly different) box. His majorant can be shown to be extremal for certain configurations of boxes, and it is unknown if the minorant is ever extremal for any given configuration of boxes. Nevertheless, such approximations have proven to be useful in applications \cite{Baker, BMV, Cochrane, DrmotaTichy1997, Harman1993, Harman1998, HKW,KL} because the best known approximations are {\it asymptotically extremal} as the Fourier support becomes uniformly large. We will use Selberg's method of constructing minorants in what follows (see Proposition \ref{prop-selberg}).

We first study the Beurling-Selberg problem for multivariate Gaussian functions.\footnote{This problem, for a single variable, was solved in \cite{CLV}.} We then generalize the {\it Gaussian subordination} and {\it  distribution method}, originally developed in \cite{CLV}, to higher dimensions and apply the method to study the Beurling-Selberg problem for a class of radial functions.  We conclude our investigations with some applications to Hilbert-type inequalities and adapt the construction to periodic functions.

To state the first of our main results we will use the following notation (see \S \ref{prelim} for additional information). For $c>0$, let $g_{c}(t)=e^{-c\pi t^{2}}$ and for  $a,\lambda\in\Lambda$ ( where $\Lambda=(0,\infty)^d$) let 
\[
	x\in\R^d\mapsto G_{\lambda}(x)=\dprod_{j=1}^{d}g_{\lambda_{j}}(x_{j})=\exp\lb -\dsum_{j=1}^{d}\lambda_{j}\pi x_{j}^{2}  \rb.
\]
and let 
\[
	Q(a)=\dprod_{j=1}^{d}[-a_{j},a_{j}].
\]
\indent Our first result is a solution to the Beurling-Selberg extremal problem of determining optimal {\it majorants} of the Gaussian function $G_{\lambda}(x)$ that have Fourier transform supported in $Q(a)$.
\begin{theorem}\label{maj-theo}
Let $a,\lambda\in\Lambda$. If $F:\R^d\to\R$ is an integrable function that satisfies 
\begin{enumerate}[(i)]
\item $F(x)\geq G_\lambda(x)$ for each $x\in\R^d$, and 
\item $\hat{F}(\xi)=0$ for each $\xi\not\in Q(a)$,
\end{enumerate}
then
\begin{equation} \label{maj-ans}
\int_{\R^d} F(x)dx \geq \prod_{j=1}^d\lambda_j^{-\hf}\Theta(0;ia_j^2/\lambda_j).
\end{equation}
where $\Theta(v;\tau)$ is Jacobi's theta function (see \S\ref{prelim}). Moreover, equality holds if $F(z)=M_{\lambda,a}(z)$ where $M_{\lambda, a}(z)$ is defined by (\ref{maj-def}).
\end{theorem}
This theorem is essentially a corollary of Theorem 3 of \cite{CLV}. The proof simply uses the product structure and positivity of $G_{\lambda}(x)$ in conjunction with Theorem 3 of \cite{CLV}.  It would be interesting to determine the analogue of the above theorem for {\it minorants} of $G_{\lambda}(x)$ (i.e. the high dimensional analogue of Theorem 2 of \cite{CLV}) where the extremal functions cannot be obtained by a tensor product of lower dimensional extremal functions.  

In our second result we address this problem by constructing minorants  of the Gaussian function $G_{\lambda}(x)$ that have Fourier transform supported in $Q(a)$ and that are asymptotically extremal as $a$ becomes uniformly large in each coordinate.

\begin{theorem} \label{min-theo}
Let $a,\lambda\in\Lambda$. If $F:\R^d\to\R$ is an integrable function that satisfies 
\begin{enumerate}[(i)]
\item $F(x)\leq G_\lambda(x)$ for each $x\in\R^d$, and 
\item $\hat{F}(\xi)=0$ for each $\xi\not\in Q(a)$,
\end{enumerate}
then
\begin{equation} \label{min-bound}
\int_{\R^d} F(x)dx \leq \prod_{j=1}^d\lambda_j^{-\hf}\Theta(\hf;ia_j^2/\lambda_j).
\end{equation}
Furthermore, there exist a positive constant $\gamma_0=\gamma_0(d)$ such that if $\gamma:=\min\{a_j^2/\lambda_j : 1\leq j\leq d\}\geq \gamma_0$, then
\begin{equation} \label{min-asy}
\prod_{j=1}^d\lambda_j^{-\hf}\Theta(\hf;ia_j^2/\lambda_j) \leq (1+5de^{-\pi\gamma})\int_{\R^d}L_{\lambda,a}(x)dx
\end{equation}
where $L_{\lambda,a}(x)$ is defined by (\ref{min-def}).
\end{theorem}
\medskip

\begin{remark}
For a fixed $a$, if $\lambda$ is large enough then the right hand side (RHS) of (\ref{min-asy}) would be negative and the inequality would not hold, but this is not true for large values of $\gamma$. In fact, this happens because the zero function would be a better minorant.
\end{remark}

\begin{remark} Inequality (\ref{min-asy}) implies that if $\lambda$ is fixed and $a$ is large, then  $\int_{\R^d}L_{\lambda,a}(x)dx$ approaches exponentially fast to the optimal answer. In this sense, we say that $L_{\lambda,a}(x)$ is asymptotically optimal with respect to the type.
\end{remark}

Our next set of results (which are too lengthy to state here) are Theorems \ref{existence}, \ref{gauss-sub-maj}, and \ref{gauss-sub-min}. See \S \ref{Gaussian Subordination Method} for the statements of the theorems. These theorems generalize the so-called {\it distribution} and {\it Gaussian subordination} methods of \cite{CLV}. The main idea behind these methods goes back to the paper of Graham and Vaaler \cite{GV}. We will describe a ``watered down'' version of the approach here. Let us begin with the inequality
\[
	G_{\lambda}(x)\leq F_{\lambda}(x)
\]
where $F_{\lambda}(x)$ is defined by (\ref{maj-def}). The idea is to {\it integrate the free parameter} $\lambda$ in the function $G_{\lambda}(x)$ with respect to a (positive) measure $\nu$ on $\Lambda=(0,\infty)^{d}$ to obtain a pair of new functions of $x$:
\[
		g(x)=\dint_{\Lambda} G_{\lambda}(x)d\nu(\lambda)\leq \dint_{\Lambda} F_{\lambda}(x)d\nu(\lambda)=f(x).
\]
The process simultaneously produces a function $g(x)$ and a majorant $f(x)$ having $\hat{f}(\xi)$ supported in $Q(a)$. The difference of the functions in $L^{1}-$norm is similarly obtained by integrating against $\nu$. The method that we present allows us to produce majorants and minorants for $g(x)$ equal to the one of the following functions (among others):
\begin{eqnarray*}
g(x) &=& e^{{-\alpha}|x|^{r}}, \ \mathrm{ for }  \;\alpha>0 \ \ and\ \ 0<r<2, \\
g(x) &=& (|x|^{2}+\alpha^{2})^{-\beta}, \mathrm{ for }  \ \alpha>0\ and\ \beta>0, \\
g(x) &=& -\log\left(\frac{|x|^{2}+\alpha^{2}}{|x|^{2}+\beta^{2}}\right), \mathrm{ for } \ 0<\alpha<\beta, \; \mathrm{ and }\\
g(x) &=& |x|^{\sigma}, \mathrm{ for } \ \sigma\in(0,\infty)\setminus 2\Z_+
\end{eqnarray*}
where $|x|$ is the Euclidean norm of $x$ (see \S\ref{prelim}). In \S\ref{classes}, we will discuss the full class of functions for which our method produces majorants and minorants. 

One dimensional extremal functions have proven to be useful in several problems in analytic number theory \cite{CCLM,CL3,CC,MR3063902,CLV,MR2781205,MR809967,MR2331578,GV,Harman1998,M1978,M2,S1,S2,V}. The most well known examples make use of Selberg's extremal functions for the characteristic function of the interval. These include sharp forms of the large sieve and the Erd{\H o}s-Tur\'an inequality. Recent activity include estimates of quantities related to the Riemann zeta function. These include estimates of the zeta function on vertical lines in the critical strip and estimates of the pair correlation of the zeros of the zeta function. It would be interesting to see similar applications in the multidimensional setting. 

\subsection*{Organization of the paper}
In \S \ref{prelim} we introduce notation and gather the necessary background material for the remainder of the paper. Then in \S\ref{M.G.F} we will discuss and prove Theorems \ref{maj-theo} and \ref{min-theo}. Next, in \S\ref{Gaussian Subordination Method} we present and prove a generalization of the {\it distribution method} introduced in \cite{CLV} as well as a generalization of their Gaussian subordination result. In \S\ref{per-ana} we study periodic analogues of Theorems \ref{maj-theo}, \ref{min-theo}, \ref{existence}, \ref{gauss-sub-maj}, and \ref{gauss-sub-min}. We conclude with some applications to Hilbert-type inequalities \S\ref{Hilbert} and some final remarks in \S\ref{conclusion}

\section{Preliminaries}\label{prelim}
Let us first have a word about our notation. Throughout this manuscript vectors in $\R^d$ will be denoted by lowercase letters such as $x=(x_{1},...,x_{d})$, and the Euclidean norm of $x$ is given by
\[
	|x|=\lb x_{1}^{2}+\cdots+x_{d}^{2}  \rb^{1/2}.
\]
 We will write $\T^d=\R^d/\Z^d$, $\Lambda=(0,\infty)^d$ and for each $a\in\Lambda$ we let
\begin{equation}\label{parall}
Q(a)=\{x\in\R^d : |x_j|\leq a_j\}.
\end{equation}
The vector $u=(1,...,1)\in\Lambda$ will be called the {\it unitary} vector and we will use the notation $Q(R)=Q(Ru)$ whenever $R>0$ is a positive real number. \newline
\indent Now we will introduce some notation that is not standard, but convenient for our purposes. Given vectors $x,y\in\R^d$ we will write $xy=(x_{1}y_{1},...,x_{d}y_{d})$ and if $y_{j}\neq0$ for each $j$, then we will write $x/y=(x_{1}/y_{1},...,x_{d}/y_{d})$. We say that $x<y$ $(x\leq y)$ if $x_{j}<y_{j} \  \forall j$ $(x_{j}\leq y_{j}\ \forall j)$. We will always denote the inner product  of $x,y\in\R^d$ by a central dot, that is, $x\cdot y$. \newline

One of the main objects of study in this paper is the {\it Fourier transform}. Given an integrable function $F(x)$ on $\R^d$, we define the Fourier transform of $F(x)$ by
\[
	\hat{F}(\xi)=\dint_{\R^d}e(x\cdot\xi)F(x)dx
\]
where $\xi\in\R^d$ and $e(\theta)=e^{-2\pi i \theta}$. We extend the definition in the usual way to tempered distributions (see for instance \cite{SW}). We will mainly be considering functions whose Fourier transforms are supported in a bounded subset of $\R^d$. Such functions are called {\it band-limited}. It is well-known that band-limited functions can be extended to an entire functions on $\C^d$ satisfying an exponential growth condition. An entire function $F:\C^d\ra \C$ is said to be of {\it exponential type} if there exists a number $\tau\geq0$ such that for every $\epsilon>0$ there exists a constant $C_{\epsilon}>0$ such that
\[
	|F(z)|\leq C_{\epsilon}e^{(\tau+\epsilon)|z|}.
\]
If $F(z)$ satisfies such a growth estimate, then $F(z)$ is said to be of exponential type at most $\tau$.\newline
\indent There is a refinement of the definition of exponential type due to Stein \cite{S57,SW}. Given an origin symmetric convex body $K$ with supporting function 
\begin{equation}\label{supp-func}
H(z)=H_{K}(z)=\sup_{\xi\in K}z\cdot\xi,
\end{equation}
an entire function $F:\C^d\ra \C$ is said to be of exponential type with respect to $K$ if 
\[
	|F(z)|\leq C_{\epsilon}e^{(1+\epsilon)H(z)}.
\]
This is the natural generalization of exponential type used by Stein in his generalization of the Paley-Wiener theorem \cite{S57, SW}. The Paley-Wiener theorem has another generalization to higher dimensions that is formulated for certain tempered distributions. In this formulation, it is not necessary that the body $K$ be symmetric. We will now state this generalization of the Paley-Wiener than which can be found in \cite{Hor}, Theorem 7.3.1.

\begin{theorem}[Paley--Wiener--Schwartz]\label{thm-PWS} 
Let K be a convex compact subset of $\R^d$ with supporting function $H(x)$ given by (\ref{supp-func}). If $F$ is a tempered distribution such that the support of $\ft{F}$ is contained in K, then $F:\C^d\to\C$ is an entire function and exist $N,C>0$ such that
$$
|F(x+iy)| \leq C(1+|x+iy|)^Ne^{2\pi H(y)}.
$$
for every  $x+iy\in\C^d$. 

Conversely, every entire function $F:\C^d\to\C$ satisfying an estimate of this form defines a tempered distribution with Fourier transform supported on $K$.
 
\end{theorem}
\medskip

We will now define and compile some results about Gaussians and theta functions that we will need in the sequel. Given a positive real number $\delta>0$ the Gaussian $g_{\delta}:\R\ra\R$ is defined by 
\[
	g_{\delta}(t)=e^{-\delta\pi t^{2}},
\]
and its Fourier transform is given by $\hat{g}_{\delta}(\xi)=\delta^{-1/2}g_{1/\delta}(\xi)$. 

\noindent For a $\tau=\sigma +it$ with $t>0$, if $q=e^{\pi i \tau}$, then Jacobi's theta function (see \cite{Ch}) is defined by 
\begin{equation}\label{theta-function}
\Theta(v;\tau)=\dsum_{n\in\Z}e( nv)q^{n^2}.
\end{equation}
	
These functions are related through the Poisson summation formula by
\begin{equation}\label{Ptheta}
	\dsum_{m\in\Z}g_{\delta}(v+m)=\dsum_{n\in\Z}e( nv) \hat{g}_{\delta}(n)=\delta^{-1/2}\Theta(v;i\delta^{-1}).
\end{equation}

\medskip 
 The one dimensional case of Theorems \ref{maj-theo} and \ref{min-theo} are proven in \cite{CLV}. There it is proved that the functions
 
\begin{equation}\label{extremal-min}
l_{\delta}(z)= \Big( \frac{\cos \pi z}{\pi} \Big)^2 \bigg\{ \sum_{k\in\Z}
\frac{g_{\delta}(k+\frac{1}{2})}{(z-k-\frac{1}{2})^2} + \sum_{k\in\Z}\frac{g'_{\delta}(k+\frac{1}{2})}{(z-k-\frac{1}{2})} \bigg\}
\end{equation}
and

\begin{equation}\label{extremal-maj}
m_{\delta}(z)= \Big( \frac{\sin \pi z}{\pi} \Big)^2 \bigg\{ \sum_{k\in\Z}
\frac{g_{\delta}(k)}{(z-k)^2} + \sum_{k\in\Z}\frac{g'_{\delta}(k)}{(z-k)} \bigg\}
\end{equation}
are entire functions of exponential type at most $2\pi$  and they satisfy 
\begin{equation}\label{ineq-cond}
l_{\delta}(x)\leq g_{\delta}(x)\leq m_{\delta}(x)
\end{equation}
for all real $x$. Moreover,

\begin{equation} \label{gauss-maj-ans}
\int_{-\infty}^{\infty} m_{\delta}(x)dx=\delta^{-\hf}\Theta(0;i/\delta)
\end{equation}
and

\begin{equation} \label{gauss-min-ans}
\int_{-\infty}^{\infty} l_{\delta}(x)dx=\delta^{-\hf}\Theta(\hf;i/\delta).
\end{equation}		

In view of (\ref{Ptheta}), (\ref{gauss-maj-ans}), and (\ref{gauss-min-ans}), the functions $l_{\delta}(z)$ and $m_{\delta}(z)$ are the best one-sided $L^{1}-$approximations of $g_{\delta}$ having exponential type at most $2\pi$.

\section{The Multidimensional Gaussian Function}\label{M.G.F}
In this section we will prove Theorem \ref{maj-theo} and \ref{min-theo}. To construct a minorant of the Gaussian, we begin with the functions $m_{\delta}(z)$ and $l_{\delta}(z)$ defined by (\ref{extremal-min}) and (\ref{extremal-maj}) and use Selberg's bootstrapping technique to obtain multidimensional minorants.  The majorant is constructed by taking $m_{\delta}(z)$ tensored with itself $d$-times.\newline
\indent For every $\lambda = (\lambda_1,...,\lambda_n) \in \Lambda$ we define the function
$G_\lambda:\R^d\to\R$ by
\begin{equation} \label{target-func}
G_\lambda(x)=\prod_{j=1}^dg_{\lambda_j}(x_j)=\prod_{j=1}^d e^{-\lambda_j\pi x_j^2}.
\end{equation}

\noindent The following proposition is due to Selberg, but it was never published \cite{VPC}. We called it {\it Selberg's bootstrapping method} because it enables us to construct a minorant for a tensor product of functions provided that we have majorants and minorants of each component at our disposal. The method has been used in one form or the other in \cite{BMV,DrmotaTichy1997,Harman1993,Harman1998,HKW}.
 
\begin{proposition}\label{prop-selberg}
Let $d>0$ be natural number and $f_j:\R\to (0,\infty)$ be functions for every $j=1,...,d$. Let $l_j,m_j:\R\to\R$ be real-valued functions such that 
\[
l_j(x) \leq f_j(x) \leq m_j(x)
\]
for every $x$ and $j$. Then
\[
-(d-1)\prod_{k=1}^{d}m_k(x_{k})+\sum_{k=1}^{d}l_k(x_{k})\underset{j\neq k}{\prod_{j=1}^{d}}m_j(x_{j})\leq \prod_{k=1}^{d}f_k(x_{k}).
\]
\end{proposition}

\medskip

\noindent This proposition is easily deduced from the following inequality.

\begin{lemma}\label{ineq-lemma}
If $\beta_{1},...,\beta_{d}\geq 1$, then 
\[
\sum_{k=1}^{d}\underset{j\neq k}{\prod_{j=1}^{d}}\beta_{j}\leq 1 +(d-1)\prod_{k=1}^{d}\beta_{k}.
\]
\end{lemma}
\begin{proof}
We give a proof by induction, starting with the inductive step since the base case is simple.

			\noindent Suppose that the claim is true for $d=1,...,L$. Let $\beta_1,...,\beta_L,\beta_{L+1}$ be a sequence of real numbers not less than one and write $\beta_j=1+\epsilon_j$. We obtain
\begin{eqnarray*}
				\sum_{k=1}^{L+1}\underset{j\neq k}{\displaystyle\prod_{j=1}^{L+1}}\beta_{j}
				 &=& \displaystyle\prod_{j=1}^{L}\beta_{j} +(1+\epsilon_{L+1})\sum_{k=1}^{L}\underset{j\neq k}{\displaystyle\prod_{j=1}^{L}}\beta_{j} \\
				 &\leq & \displaystyle\prod_{j=1}^{L}\beta_{j} +(1+\epsilon_{L+1})\bigg\{ 1+ (L-1)\displaystyle\prod_{j=1}^{L}\beta_{j} \bigg\} \\
				&=&  1 +\epsilon_{L+1} +	 \displaystyle\prod_{j=1}^{L}\beta_{j} + (L-1)\displaystyle\prod_{j=1}^{L+1}\beta_{j} \\
				 &\leq &  1 +\epsilon_{L+1}\displaystyle\prod_{j=1}^{L}\beta_{j}+	 \displaystyle\prod_{j=1}^{L}\beta_{j} + (L-1)\displaystyle\prod_{j=1}^{L+1}\beta_{j} \\
				 &=& 1+L\displaystyle\prod_{j=1}^{L+1}\beta_{j}
\end{eqnarray*}
\end{proof}
\smallskip

Now we can define our candidates for majorant and minorant of $G_{\lambda}(x)$. For a given $\lambda\in\Lambda$ define the functions
\begin{equation} \label{min-def}
z\in\C^d\mapsto L_\lambda(z)=-(d-1)\prod_{j=1}^d m_{\lambda_j}(z_j) + \sum_{k=1}^d l_{\lambda_k}(z_k)\underset{j\neq k}{\displaystyle\prod_{j=1}^{d}}m_{\lambda_j}(z_j)
\end{equation}
and
\begin{equation}\label{maj-def}
z\in\C^d\mapsto M_\lambda(z)=\prod_{j=1}^d m_{\lambda_j}(z_j).
\end{equation}
It follows from Proposition \ref{prop-selberg} and (\ref{ineq-cond}) that 
\begin{equation}\label{min-maj-ineq}
L_\lambda(x)\leq G_\lambda(x)\leq M_\lambda(x)\, \mbox{ for all } \, x\in\R^d.
\end{equation}
Moreover, since $l_\delta(x)$ and $m_\delta(x)$ have exponential type at most $2\pi$, we conclude that the Fourier transforms of $L_\lambda(x)$ and $M_\lambda(x)$ are supported on $Q$. We  modify $L_\lambda(z)$ and $M_\lambda(z)$ to have exponential type with respect to $Q(a)$ in the following way. Given $a,\lambda\in\Lambda$ we define  the functions
\begin{equation}\label{min-ext-def}
L_{\lambda,a}(z)=L_{\lambda/a^2}(az)
\end{equation}
and
\begin{equation}\label{maj-ext-def}
M_{\lambda,a}(z)=M_{\lambda/a^2}(az).
\end{equation}
By (\ref{min-maj-ineq}) we obtain
\begin{equation}
L_{\lambda,a}(x)\leq G_\lambda(x)\leq M_{\lambda,a}(x)\, \mbox{ for all } \, x\in\R^d
\end{equation}
and using the scaling properties of the Fourier transform, we conclude that $L_{\lambda,a}(x)$ and $M_{\lambda,a}(x)$ have exponential type with respect to $Q(a)$.
By formula (\ref{extremal-maj}), we have $m_\delta(k)=g_\delta(k)$ for all integers $k$, hence we obtain
\begin{equation}\label{maj-interp}
M_{\lambda,a}(k/a)=G_\lambda(k/a)
\end{equation}
for all $k\in\Z^d$ ( recall that $k/a=(k_1/a_1,...,k_d/a_d)$ ).\newline

%%%%%%%%%%%%%%%%%%%%%%%%%%%%%%%%%%%%%%%%%%%%%%%%%%%%%%%%

%\begin{theorem}\label{maj-theo}
%Let $a$ and $\lambda$ be vectors in $\Lambda$. Then for every real entire function $F:\C^d\to\C$ of exponential type with respect to $Q(a)$ that satisfies $F(x)\geq G_\lambda(x)$ for every $x\in\R^d$ we have
%\begin{equation} \label{maj-ans}
%\int_{\R^d} F(x)dx \geq \prod_{j=1}^d\lambda_j^{-\hf}\Theta(0;ia_j^2/\lambda_j).
%\end{equation}
%Equality holds if $F(z)=M_{\lambda,a}(z)$.
%\end{theorem}
%\begin{theorem} \label{min-theo}
%Let $a$ and $\lambda$ be vectors in $\Lambda$. Then for every real entire function $F:\C^d\to\C$ of exponential type with respect to $Q(a)$ that satisfies $F(x)\leq G_\lambda(x)$ for every $x\in\R^d$ we have
%\begin{equation} \label{min-bound}
%\int_{\R^d} F(x)dx \leq \prod_{j=1}^d\lambda_j^{-\hf}\Theta(\hf;ia_j^2/\lambda_j).
%\end{equation}
%Furthermore, there exist a positive constant $\gamma_0=\gamma_0(d)$ such that if $\gamma:=\min\{a_j^2/\lambda_j : 1\leq j\leq d\}\geq \gamma_0$, then
%\begin{equation} \label{min-asy}
%\prod_{j=1}^d\lambda_j^{-\hf}\Theta(\hf;ia_j^2/\lambda_j) \leq (1+5de^{-\pi\gamma})\int_{\R^d}L_{\lambda,a}(x)dx
%\end{equation}
%\end{theorem}
%\medskip
We are now in a position to prove Theorems \ref{maj-theo} and \ref{min-theo}.
\begin{proof}[Proof of Theorem \ref{maj-theo}] 
It follows from (\ref{min-maj-ineq}) and (\ref{maj-ext-def}) that the function $M_{\lambda,a}(x)$ is majorant of $G_\lambda(x)$ of exponential type with respect to $Q(a)$. Define $\alpha=a_1\cdots a_d$. Using definition (\ref{maj-def}) and (\ref{maj-ext-def}) we obtain
\begin{eqnarray*}
\int_{\R^d}M_{\lambda,a}(x)\dx \,\, = \,\, \alpha^{-1} \int_{\R^d}M_{\lambda/a^2}(x)\d x&=&\alpha^{-1} \prod_{j=1}^d\int_{\R^d}m_{\lambda_j/a_j^2}(x_j)dx_j  \\
 &=&\prod_{j=1}^d\lambda_j^{-\hf}\Theta(0;ia_j^2/\lambda_j),
\end{eqnarray*}
where the first equality is due to a change of variables, the second one due to the product structure and the third one due to (\ref{gauss-maj-ans}).

{\bf Now we will prove that (\ref{maj-ext-def}) is extremal.}
Suppose that $F(z)$ is an entire majorant of $G_\lambda(x)$ of exponential type with respect to $Q(a)$ and  integrable on $\R^d$ (and therefore absolutely integrable on $\R^d$). We then have

\begin{eqnarray}\label{uni-ineq}
\int_{\R^d} F(x)dx = \alpha^{-1}\sum_{k\in\Z^d}F(k/a)\geq \alpha^{-1}\sum_{k\in\Z^d} G_\lambda(k/a) = \prod_{j=1}^d\lambda_j^{-\hf}\Theta(0;ia_j^2/\lambda_j)
\end{eqnarray}
\smallskip

\noindent because $M_{\lambda,a}(x)$ majorizes $G_{\lambda},$ and the rightmost equality is given by (\ref{Ptheta}).

\end{proof}

\noindent To prove Theorem \ref{min-theo} we will need the following lemma.

% Before we turn to the proof of Theorem \ref{min-theo} we need a technical lemma that will be proven later.
\begin{lemma}\label{theta-lemma}
For all $t> 0$ we have
\begin{equation}\label{theta-ineq}
1-4q/(1-q)^2<\frac{\Theta(\hf;it)}{\Theta(0;it)} <
e^{-2q},
\end{equation}
where $q=e^{-\pi t}$.

\end{lemma}
\medskip

\begin{proof}[Proof of Theorem \ref{min-theo}]

Suppose that $F(z)$ is an entire minorant of exponential type with respect to $Q(a)$ and absolutely integrable on $\R^d$. Recalling that $u=(1,...,1)$ and applying the Poisson summation formula, we obtain

\begin{eqnarray*}\label{uni-ineq}
\int_{\R^d} F(x)dx = \alpha^{-1}\sum_{k\in\Z^d}F(k/a+u/2a)&\leq & \alpha^{-1}\sum_{k\in\Z^d} G_\lambda(k/a+u/2a) \\ &=& \prod_{j=1}^d\lambda_j^{-\hf}\Theta(\hf;ia_j^2/\lambda_j).
\end{eqnarray*}
\smallskip

\noindent where the last equality is given by (\ref{Ptheta}).
This proves (\ref{min-bound}). By construction, $L_{\lambda,a}(z)$ is an entire minorant of exponential type with respect to $Q(a)$. Using definitions (\ref{min-def}) and (\ref{min-ext-def}) we conclude that

\begin{equation}\label{L-int}
\int_{\R^d}L_{\lambda,a}(x)dx = \bigg\{ \sum_{j=1}^d\frac{\Theta(\hf;ia_j^2/\lambda_j)}{\Theta(0;ia_j^2/\lambda_j)}-(d-1)\bigg\} \prod_{j=1}^d\Theta(0;ia_j^2/\lambda_j)\lambda_j^{-\hf}.
\end{equation}
\smallskip

\noindent Thus, to deduce (\ref{min-asy}), we only need to prove that

\begin{equation}\label{vital-ineq}
(1+5de^{-\pi\gamma})\bigg\{ \sum_{j=1}^d\frac{\Theta(\hf;ia_j^2/\lambda_j)}{\Theta(0;ia_j^2/\lambda_j)}-(d-1)\bigg\} \geq \prod_{j=1}^d\frac{\Theta(\hf;ia_j^2/\lambda_j)}{\Theta(0;ia_j^2/\lambda_j)}.
\end{equation}
\smallskip

\noindent for large $\gamma$ (recall that $\gamma=\min\{\lambda_j/a_j^2\}$). If we let $q_j=e^{-\pi a_j^2/\lambda_j}$ and $\gamma$ sufficient large such that $(1-e^{-\pi\gamma})^2>4/5$, we can use Lemma \ref{theta-lemma} to obtain

\begin{equation}\label{ineq1}
\sum_{j=1}^d\frac{\Theta(\hf;ia_j^2/\lambda_j)}{\Theta(0;ia_j^2/\lambda_j)}-(d-1) \geq 1 - 4\sum_{j=1}^d q_j/(1-q_j)^2 \geq 1 - 5\sum_{j=1}^d q_j.
\end{equation}
\smallskip

\noindent Applying Lemma \ref{theta-lemma} for a sufficient large $\gamma$ we obtain

\begin{equation}\label{ineq2}
\prod_{j=1}^d\frac{\Theta(\hf;ia_j^2/\lambda_j)}{\Theta(0;ia_j^2/\lambda_j)}\leq \exp\bigg\{ -2\sum_{j=1}^dq_j\bigg\} \leq 1-\sum_{j=1}^dq_j.
\end{equation}
\smallskip

\noindent where the last inequality holds, say, if $\sum_{j=1}^dq_j<\log 2$ .
Write $\beta=\sum_{j=1}^dq_j$ and note that 
$$
1-\beta \leq (1-5\beta)(1+5\beta)
$$
if $\beta$ is sufficiently small, say, if $\beta\in[0,1/25)$. If $\gamma$ is sufficiently large such that $\sum_{j=1}^dq_j<1/25$ we obtain

\begin{equation}\label{ineq3}
1-\sum_{j=1}^dq_j \leq \big(1-5\sum_{j=1}^dq_j\big)\big(1+5\sum_{j=1}^dq_j\big) < \big(1-5\sum_{j=1}^dq_j\big)(1+5de^{-\pi\gamma}).
\end{equation}
\smallskip

\noindent By (\ref{ineq1}), (\ref{ineq2}) and (\ref{ineq3}) we conclude that there exists an $\gamma_0=\gamma_{0}(d)>0$ such that if $\gamma\geq \gamma_0$, then (\ref{vital-ineq}) holds.

\end{proof}

\begin{proof}[Proof of Lemma \ref{theta-lemma}]
Recall that $e(v)=e^{2\pi iv}$ and $q=e^{\pi i\tau}$. By Theorem 1.3 of chapter 10 of \cite{SS} the Theta function has the following product representation

\begin{equation}\label{theta-prod-rep}
\Theta(v;\tau)=\prod_{n=1}^{\infty}\big(1-q^{2n}\big)\big(1+q^{2n-1}e(v)\big)\big(1+q^{2n-1}e(-v)\big).
\end{equation}
\smallskip

\noindent It follows from (\ref{theta-prod-rep}) that

$$
\frac{\Theta(\hf;it)}{\Theta(0;it)}=\prod_{n=1}^{\infty}\bigg(\frac{1-q^{2n-1}}{1+q^{2n-1}}\bigg)^2 = \exp\bigg\{ 2\sum_{n=0}^{\infty}\log\bigg(1-\frac{2q^{2n+1}}{1+q^{2n+1}}\bigg)\bigg\}.
$$
\smallskip

\noindent Using the inequality $\log(1-x)\geq -x/(1-x) $ for all $x\in[0,1)$ we obtain

\begin{eqnarray*}
\frac{\Theta(\hf;it)}{\Theta(0;it)}\geq \exp\bigg\{-4\sum_{n=0}^{\infty}\frac{q^{2n+1}}{1-q^{2n+1}}\bigg\}&\geq &\exp\bigg\{-\frac{4}{1-q}\sum_{n=0}^{\infty}q^{2n+1}\bigg\}\\ &=&\exp\bigg\{-\frac{4q}{(1-q)(1-q^2)}\bigg\} \\ &>& e^{-4q/(1-q)^2} \\ &>& 1-4q/(1-q)^2
\end{eqnarray*}
\smallskip

\noindent and this proves the left hand side (LHS) inequality in (\ref{theta-ineq}). The (RHS) inequality in (\ref{theta-ineq}) is deduced by a similar argument using the inequality $\log(1-x)\leq -x $ for all $x\in[0,1)$.

\end{proof}

%%%%%%%%%%%%%%%%%%%%%%%%%%%%%%%%%%%%%%%%%%%%%%%%%%%%%%%%%%%%

\section{Gaussian Subordination Method}\label{Gaussian Subordination Method}
In this section we adapt the distribution method developed in \cite{CLV} to the several variables setting, and we apply the Gaussian subordination method to the majorant and minorant for the multidimensional Gaussian defined in the previous section. This method allow us to extend the class of function $g(x)$ for which we can solve the corresponding Beurling--Selberg extremal problems. The central idea is to integrate the functions $G_\lambda(x)$, $L_{\lambda,a}(z)$ and $M_{\lambda,a}(z)$, in a distributional sense, with respect to a non--negative measure $\nu$ defined on $\Lambda=(0,\infty)^d$.

We begin with a generalization of the distribution method developed in \cite{CLV} for existence of majorants and minorants, we will deal with extremality later.

\begin{theorem}[Distribution Method -- (Existence)]\label{existence}
Let $K\subset \R^d$ be a compact convex set, $\Lambda$ be a measurable space of parameters, and for each ${\lambda}\in \Lambda$ let $G({x};{\lambda})\in L^1(\R^d)$ be a real-valued function. For each $\lambda$ let $F({z};{\lambda})$ be a entire function defined for $z\in\C^d$ of exponential type with respect to $K$. Let $\nu$ be a non--negative measure on $\Lambda$ that satisfies
 					\begin{equation}\label{intValue}
							\int_{\Lambda}\int_{\R^d} | F({x};{\lambda})-G({x};{\lambda}) | \d{x} \d\nu({\lambda})   <\infty.
					\end{equation}
					and
					\begin{equation}\label{Fourier-cond-nu}
					\int_{\Lambda}\int_{\R^d}|\ft{G}({x};{\lambda})\varphi({x})| \d{x} \d\nu({\lambda})<\infty
					\end{equation}
					for all $\varphi\in{\mc S}(\R^d)$ supported in $K^{c}$.
					
Let $\G\in{\mc S}^{\prime}(\R^d)$ be a real-valued continuous function and
					\begin{equation}\label{gFourier1}
							\ft{\G}(\varphi)=\int_{\R^d}  \int_{\Lambda}\ft{G}({x};{\lambda})d\nu({\lambda}) \varphi({x})d{x}
					\end{equation}
				for all $\varphi\in{\mc S}(\R^d)$ supported in $K^{c}$.
				
			\begin{enumerate}
				\item[(i)] If  $  G({x};{\lambda})\leq F({x};{\lambda})$ for each ${x}\in \R^d$ and ${\lambda}\in \Lambda$, then there exists a real entire majorant $ \M({x})$ for $\G({x})$ of exponential type with respect to $K$ and 
				\[
					\int_{\R^d} \{ \M({x})-\G({x})\} d{x}
				\] 
				 is equal to the quantity at (\ref{intValue}).
				\item[(ii)] If  $  F({x};{\lambda}) \leq G({x};{\lambda})$ for each ${x}\in \R^d$ and ${\lambda}\in \Lambda$, then there exists a real entire minorant $ \LL({x})$ for $\G({x})$ of exponential type with respect to $K$ and 
				\[
					\int_{\R^d}\{ \G({x})-\LL({x}) \}d{x}
				\] 
				is equal to the quantity at (\ref{intValue}).
			\end{enumerate} 
 		\end{theorem} 

\medskip
			
\noindent With the exception of Theorem \ref{existence}, $\Lambda$ will always stand for $(0,\infty)^d$. For a given $a\in\Lambda$ define $\GG^d_+(a)$ as the set of ordered pairs $(\G,\nu)$ where $\G:\R^d\to\R$ is a continuous function and $\nu$ is a non-negative Borel measure in $\Lambda$ such that
\smallskip

\begin{description}
	\item[(C1) ] $\G\in \mc S'(\R^d)$ is a tempered distribution.
	\item[(C2)]For all $\varphi\in{\Sw}(\R^d)$ supported in $Q(a)^c$ we have
$$
\int_{\R^d}\int_{\Lambda}|\ft{G}_\lambda({x})\varphi({x})|d\nu({\lambda}) d{x} <\infty.
$$
	\item[(C3)] For all $\varphi\in{\Sw}(\R^d)$ supported in $Q(a)^c$ we have

\begin{equation}\label{G-Fourier}
\int_{\R^d}\G(x)\ft{\varphi}(x)=\int_{\R^d}\int_{\Lambda}\ft{G}_\lambda(x)\d\nu(\lambda)\varphi(x)dx.
\end{equation}
	\item[(C4+)] The following integrability condition holds

\begin{equation}\label{maj-nu-cond}
						\int_{\Lambda} \prod_{k=1}^{d}\lambda_{k}^{-\hf}\bigg\{ \prod_{j=1}^{d}\Theta(0;ia_j^2/\lambda_j) - 1\bigg\}    d\nu({\lambda}) <\infty.
					\end{equation}
\end{description}
\noindent In a analogous way, we define the class $\GG^d_-(a)$ replacing condition {\bf (C4+)} by

\begin{description}

	\item[(C4-)] The following integrability condition holds

\begin{equation}\label{min-nu-cond}
 					 \int_\Lambda \bigg\{1- \bigg( \sum_{j=1}^d\frac{\Theta(\hf;ia_j^2/\lambda_j)}{\Theta(0;ia_j^2/\lambda_j)}-(d-1)\bigg) \prod_{j=1}^d\Theta(0;ia_j^2/\lambda_j)\bigg\}\prod_{j=1}^d\lambda_j^{-\hf}\d\nu(\lambda)
<\infty.
 					\end{equation}
\end{description}

\noindent Theorem \ref{gauss-sub-maj} offers a optimal resolution of the Majorization Problem for the class functions $\GG^d_+(a)$ and the Theorem \ref{gauss-sub-min} offers asymptotically optimal resolution of the Minorization Problem for the class of functions $\GG^d_-(a)$.
\medskip

		\begin{theorem}[Gaussian Subordination -- Majorant]\label{gauss-sub-maj}
		For a given $a\in\Lambda$, let $(\G,\nu)\in\GG^d_+(a)$. Then there exists an extremal majorant $\M_a({ z})$ of exponential type with respect to $Q(a)$ for $\G({ x})$. Furthermore, $\M_a({ x})$ interpolates $\G({ x})$ on $\Z^d/a$ and satisfies
		
		\begin{equation}\label{maj-ext-ans}
		\int_{\R^d}\M(x)-\G(x)\dx=\int_{\Lambda} \prod_{k=1}^{d}\lambda_{k}^{-\hf}\bigg\{ \prod_{j=1}^{d}\Theta(0;ia_j^2/\lambda_j) - 1\bigg\}    d\nu({\lambda}).
		\end{equation}

 		\end{theorem} 
\smallskip

\begin{theorem}[Gaussian Subordination -- Minorant]\label{gauss-sub-min}
For a given $a\in\Lambda$, let $(\G,\nu)\in\GG_-^d(a)$. Then, if $\F(z)$ is a real entire minorant of $\G(x)$ of exponential type with respect to $Q(a)$, we have

				\begin{equation}\label{G-min-bound}
				\int_{\R^d}  \G({ x})-\F(x)  d{ x}\geq \int_{\Lambda} \prod_{j=1}^{d}\lambda_{k}^{-\hf}\bigg\{ 1-\prod_{j=1}^{d}\Theta(\hf;ia_j^2/\lambda_j)\bigg\}    d\nu({\lambda}).
				\end{equation}
				
\noindent Furthermore, there exists a family of minorants $\{\LL_a({ z}):a\in\Lambda\}$ where $\LL_a(z)$ is of exponential type with respect to $Q(a)$ such that

$$
\int_{\R^d}  \G(x)-\LL_{a}({ x})   \d{ x}
$$

\noindent is equal to the (LHS) of (\ref{min-nu-cond}). Also

\begin{equation}\label{L1-min-conv}
\lim_{a\uparrow\infty}\int_{\R^d}  \G(x)-\LL_{a}({ x})   \d{ x}=0,
\end{equation}
\smallskip
					
					\noindent where $a\uparrow\infty$ means $a_j\uparrow\infty$ for each $j$.
\end{theorem}

\smallskip

\begin{corollary}\label{cor1}
Under the hypothesis of Theorem \ref{gauss-sub-min}, suppose also that exists an $R>0$ such that $\supp(\nu)\subset \Lambda\cap Q(R)$, $\G\in L^1(\R^d)$ and

$$
\int_{\R^d} \G(x)\dx = \int_\Lambda \prod_{j=1}^d\lambda_j^{-\hf}\d\nu(\lambda)<\infty.
$$
\smallskip

\noindent Then, there exist a constant $\alpha_0>0$ such that, if $\alpha:=\min\{a_j\}\geq \alpha_0$ and if $\F(x)$ is a real entire minorant of $\G(x)$ of exponential type with respect to $Q(a)$, then

$$
\int_{\R^d} \F(x)\dx \leq (1+5de^{-\pi\alpha^2/R})\int_{\R^d} \LL_a(x)\dx.
$$

\end{corollary}	

\begin{proof}[Proof of Corollary \ref{cor1}]
By a direct application of Theorem \ref{min-theo} we obtain

$$
\int_{\R^d}  \F(x)  \d{ x}\leq \int_{\Lambda_R} \prod_{j=1}^{d}\Theta(\hf;ia_j^2/\lambda_j)\lambda_j^{-\hf}  \d\nu({\lambda}).
$$
\smallskip

\noindent If we choose $\gamma_0>0$, as in Theorem \ref{min-theo} and define $\alpha_0=R\gamma_0$, we can use inequality (\ref{min-asy}) to conclude that

$$
\prod_{j=1}^{d}\Theta(\hf;ia_j^2/\lambda_j)\lambda_j^{-\hf} \leq (1+5de^{-\pi\alpha^2/R})\int_{\R^d}L_{\lambda,a}(x)\dx.
$$
\smallskip

\noindent for all $\lambda\in\Lambda\cap Q(R)$, if each $a_j\geq \alpha_0$.
If we integrate this last inequality with respect to $\d\nu(\lambda)$ we obtain the desired result.

\end{proof}

\subsection{Proofs of theorems \ref{existence}, \ref{gauss-sub-maj} and \ref{gauss-sub-min}}

\begin{proof}[Proof of theorem \ref{existence}]
		We follow the proof of Theorem 14 from \cite{CLV} proving only the majorant case, since the minorant case is nearly identical. Let 
				\[
					D({x};{\lambda})=F({x};{\lambda})-G({x};{\lambda})\geq 0.
				\]
			By condition (\ref{intValue}) and Fubini's theorem the function
				\[
					\D({x})=\int_{\Lambda}D({x};{\lambda})\d\nu({\lambda})\geq 0,
				\]
				is defined for almost all $x\in\R^d$ and $ \D\in L^{1}(\R^d)$.
			The Fourier transform of ${\D}({ x})$ is a continuous function given by
				\begin{equation}\label{dFourier}  
					\ft{\D}({\xi})=\int_{\Lambda}\ft{D}({\xi};{\lambda})\d\nu({\lambda}),
				\end{equation}
			and, due to (\ref{Fourier-cond-nu}), for almost every $\xi\not\in K$ we have the alternative representation
				\begin{equation}\label{dOut}  
					\ft{\D}({\xi})=-\int_{\Lambda}\ft{G}({\xi};{\lambda})\d\nu({\lambda}).
				\end{equation}
			Let $\M$ be the tempered distribution given by 
				\begin{equation*}
					\M(\varphi)=\int_{\R^d}  \{ \D({x})+\G({x}) \} \varphi({x})d{x}.
				\end{equation*}
			Now for any $\varphi\in{\mc S}(\R^d)$ supported in $K^c$, we have by combining (\ref{gFourier1}) and (\ref{dOut}) 
				\begin{equation}
					\ft{\M}(\varphi)=\ft{\D}(\varphi)+\ft{\G}(\varphi)= 0.
				\end{equation}
			Hence $\ft{\M}$ is supported on $K$, in the distributional sense. By the Theorem \ref{thm-PWS}, it follows that the distribution $\M$ is identified with a analytic function $\M:\C^d\to\C$ of exponential type with respect to $K$ and that 
			\begin{equation}\label{defL}
				\M(\varphi) = \int_{\R^d} \M({x})\varphi({x})d{x}  
			\end{equation}
			for every $\varphi\in{\mc S}(\R^d)$. It then follows from the definition of $ \M$ and (\ref{defL}) that for almost every ${x}\in\R^d$
				\[  
					\M(x)=\D(x)+\G(x),
				\]
			which implies  $  \M({x})\geq \G({x}) $ for all ${x}\in\R^d$ since $\G({x})$ is continuous, and
			
				\[
					\int_{\R^d}   \{\M(x)-\G(x)\}dx=  \int_{{\Lambda}} \int_{\R^d}   \{F({x};{\lambda})- G({x};{\lambda}) \}  d{x}d\nu({\lambda}) <\infty.
				\]
				
\end{proof}

\noindent Now we turn to the proof of Theorem \ref{gauss-sub-maj}.
\begin{proof}[Proof of Theorem \ref{gauss-sub-maj}]	
			We follow the proof of Theorem 14 from \cite{CLV}, skipping some parts but including changes needed for higher dimensions. 

			By conditions (C1),(C2),(C3) and (C4+) we are at the position of applying Theorem \ref{existence} for the functions $M_{\lambda,a}(z)$ defined at the previous section.
Let $\M_a(z)$ be the majorant given by Theorem \ref{existence} part (i), for $K=Q(a)$, $G_\lambda(x)=G(x;\lambda)$ and $F(z;\lambda)=M_{\lambda,a}(z)$. First, we show that $ \M_a({ n/a})=  \G({ n/a})$ for each ${ n}\in\Z^d$ and then we will conclude that $\M_a(z)$ is extremal. Let $  \D(x):=:\M_a(x)-\G(x)$, by Theorem \ref{existence} we know that $\D\in L^{1}(\R^d)$. Define $\alpha=a_1a_2...a_d$ and 
			
				\[ 
					 P({ x})=\frac{1}{\alpha}\sum_{{ n}\in \Z^d}  \D({ (x+n)/a}).
				\]
				\smallskip
				
			\noindent It follows from Fubini's theorem that $ P({ x})$ is defined almost everywhere, is integrable on $\T^d=\R^d/\Z^d$ and $  \ft{P}({ k})=\ft{\D}({ ak})$ for all $k\in\Z^d$. Therefore, we have the following identity
			
				\begin{equation}\label{eq-fejer}  
				  P*F_{R}({ x})=\sum_{\stackrel{n\in\Z^d}{|n_j|\leq R}}  \prod_{j=1}^{d} \bigg( 1-\frac{|n_{j}|}{  R +1} \bigg)  \ft{\D}({a n}) e(  {   x\cdot n}  )
				\end{equation} \smallskip
				
\noindent			for each positive integer $R$ and ${ x}\in\R^d$, where 

				\begin{equation}\label{fejer-ker}
						F_{R}({ x})=\prod_{j=1}^{d}   \frac{1}{R+1}   \bigg( \frac{\sin\pi( R+1)x_{j}}{\sin\pi x_{j}}  \bigg)^{2}=\sum_{\stackrel{n\in\Z^d}{|n_j|\leq R}}  \prod_{j=1}^{d} \bigg( 1-\frac{|n_{j}|}{  R +1} \bigg)  e(  {   x\cdot n}  )
				\end{equation} \smallskip
				
\noindent			is the product of one--dimensional Fej\'er kernels. From (\ref{dOut}) and (\ref{eq-fejer}) we have

				\begin{eqnarray*}
					P*F_{R}(0)
					&=& \ft{\D}(0) -\sum_{\stackrel{n \neq 0}{|n_j|\leq R} }  \prod_{j=1}^{d}\bigg( 1-\frac{|n_{j}|}{R+1} \bigg)   \int_{{\Lambda}} \ft{G}_\lambda({ an})d\nu({\lambda})   \\
					&=&  \ft{\D}(0) -\int_{{\Lambda}}\bigg\{\sum_{\stackrel{n \neq 0}{|n_j|\leq R} } \prod_{j=1}^{d}\bigg( 1-\frac{|n_{j}|}{R+1} \bigg)    \ft{G}_\lambda({a n})\bigg\}d\nu({\lambda}).
				\end{eqnarray*}\smallskip
				
			\noindent Since $P*F_R$ is a non--negative function and the term in the brackets above is positive we may apply Fatou's lemma to obtain
			
				\begin{eqnarray}
					   \ft{\D}(0)
					&\geq & \liminf_{ R\to \infty}    P*F_{R}(0) 
					+ \int_{{\Lambda}}  \liminf_{ R\to \infty} \bigg\{ \sum_{\stackrel{n \neq 0}{|n_j|\leq R} } \prod_{j=1}^{d}
					\bigg( 1-\frac{|n_{i}|}{R+1} \bigg)    \ft{G}_\lambda({an})  \bigg\}d\nu(\lambda)\nonumber \\ 
					&=& \liminf_{ R\to \infty}    P*F_{R}(0)
					+ \int_{\Lambda}  \sum_{n \neq 0}  \ft{G}_\lambda({an})  d\nu({\lambda})  \label{eq-fatou}.
				\end{eqnarray} \smallskip
				
\noindent 			From the Poisson Summation formula and the fact that ${M}_{\lambda,a}(n/a)={G}_\lambda(n/a)$ for all $n\in\Z^d$, we have

				\[
				\ft{M}_{\lambda,a}(0)-\ft{G}_\lambda(0)=\sum_{\stackrel{n\in\Z^d}{n\neq 0}} \ft{G}_\lambda({ an}).
				\]
				
\noindent			By (\ref{dFourier}) we  conclude that second term on the (RHS) of (\ref{eq-fatou}) is equal to $\ft{\D}(0)$. This implies that
				\[
					\liminf_{ R\to \infty}   P*F_{R}(0) \leq 0 \, \, \Longrightarrow \, \, \, \liminf_{ R\to \infty} P*F_{R}(0)=0.
				\]
				
\noindent			Using this fact with the definition of $P( { x})$, we have
				\begin{eqnarray*}
					\liminf_{R\to \infty}  P*F_{R}(0) 
					&=&  \liminf_{ R\to \infty}   \int_{\hf Q} P({- y})F_{R}({ y})d{ y}  \\
					&=& \liminf_{R\to \infty}  \int_{\hf Q}   \sum_{{ n}\in\Z^d/a}  \frac{1}{\alpha} \D((n-y)/a)          F_{R}({ y})d{ y}  \\
					&=&  \liminf_{ R\to \infty}  \sum_{{ n}\in\Z^d/a}   \int_{\hf Q}    \frac{1}{\alpha} \D((n-y)/a)   F_{R}({ y})d{ y}   \\
					&\geq&  \sum_{{ n}\in\Z^d}  \liminf_{ R\to \infty}  \int_{\hf Q}    \frac{1}{\alpha} \D((n-y)/a)   F_{R}({ y})d{ y}  \\
					&=& \sum_{{ n}\in\Z^d}     \frac{1}{\alpha} \D({ n/a}),
				\end{eqnarray*}\smallskip
				
		\noindent	where we have used the positivity of $ \D(x)$ and $ F_{R}(x)$, Fubini's theorem and Fatou's lemma. The last equality is due to F\'ejer's theorem for the continuity of $\D(x)$ at the lattice points $\Z^d/a$ (see section 3.3 of \cite{Gr}). Recalling that $ \D(x)\geq0$ for each $x\in\R^d$, it implies that $ \D({ n/a})=0$ which in turn implies $ \M_a({ n/a})=\G({ n/a})$ for each ${ n}\in\Z^d$. 
			
			To conclude, note that if $\F(z)$ is an real entire majorant of exponential type with respect to $Q(a)$ such that $\F-\G \in L^1(\R^d)$ then $\F-\M_a \in L^1(\R^d)$. Thus, the following Poisson summation formula holds pointwise
$$
\int_{\R^{d}}\{\F(x)-\M_a(x)\}dx=\frac{1}{\alpha}\sum_{n\in\Z^d}\{\F((n+y)/a)-\M_a((n+y)/a)\}
$$
for every $y\in\R^d$.
If we take $y=0$ and use that $\M_a(z)$ interpolates $\G(x)$ at the lattice $\Z^d/a$ we conclude that $\M_a(z)$ is extremal, and this concludes the theorem.

\end{proof}	

\noindent Before we turn to the proof of Theorem \ref{gauss-sub-min}, we need a technical lemma, we present the proof of that later.

\begin{lemma} \label{lemma-tech-2}
The functions
$$
\phi_a:\lambda\in\Lambda\mapsto \bigg\{ \sum_{j=1}^d\frac{\Theta(\hf;ia_j^2/\lambda_j)}{\Theta(0;ia_j^2/\lambda_j)}-(d-1)\bigg\} \prod_{j=1}^d\Theta(0;ia_j^2/\lambda_j)
$$
indexed by $a\in\Lambda$ satisfy
\begin{enumerate}
\item[(i)] $\phi_a(\lambda)\leq \phi_b(\lambda)$ if $a_j\leq b_j$
 for all $j\in\{1,...,d\}$

\item[(ii)] For all $\lambda\in\Lambda$, we have
$$
\lim_{a\uparrow\infty}\phi_a(\lambda)=1
$$
where $a\uparrow\infty$ means that $a_j\uparrow\infty$ for all j.
 \end{enumerate}
\end{lemma}
\smallskip

\begin{proof}[Proof of Theorem \ref{gauss-sub-min}]
First we prove (\ref{G-min-bound}) and then we conclude with a proof of (\ref{L1-min-conv}). 

Let $\F(z)$ be a real entire minorant of $\G(x)$ with exponential type with respect to $Q(a)$. We can assume that $\D(x)=\G(x)-\F(x) \in L^1(\R^d)$ otherwise (\ref{G-min-bound}) is trivial. Define $\alpha=a_1...a_d$. By Fubini's theorem the function
$$
P(x)=\frac{1}{\alpha}\sum_{n\in\Z^d}\D((x+n)/a)
$$
is defined almost everywhere, integrable over $\hf Q$ and $\ft{P}(n)=\ft{\D}(an)$ for all $n\in\Z^d$. Let $F_R(x)$ be the multidimensional F\'ejer kernel as defined in (\ref{fejer-ker}), hence we have the following equality
\begin{equation}
				  P*F_{R}({ x})=\sum_{\stackrel{n\in\Z^d}{|n_j|\leq R} } \prod_{j=1}^{d} \bigg( 1-\frac{|n_{j}|}{  R } \bigg)  \ft{\D}(an) e(  {   x\cdot n}  ).
\end{equation}
By condition (C3) we obtain
\begin{eqnarray*}
					\ft{\D}(0)
					&=&P*F_{R}(u/2)  -\sum_{\stackrel{n \neq 0}{|n_j|\leq R} } \prod_{j=1}^{d}\bigg( 1-\frac{|n_{j}|}{R+1} \bigg)   \int_{{\Lambda}} \ft{G}_\lambda({ an})d\nu({\lambda}) (-1)^{u\cdot n}\\
					&=&  P*F_{R}(u/2) +\int_{{\Lambda}}\bigg\{-\sum_{\stackrel{n \neq 0}{|n_j|\leq R} } \prod_{j=1}^{d}\bigg( 1-\frac{|n_{j}|}{R+1} \bigg)    \ft{G}_\lambda({a n})(-1)^{u\cdot n}\bigg\}d\nu({\lambda}).
				\end{eqnarray*}
Since $x\mapsto\ft{G}_{\lambda}(x)$ is a product of radially decreasing functions we easily see that the term in the brackets is positive, thus we can apply Fatou's lemma to obtain
\begin{eqnarray}
					   \ft{\D}(0)
					&\geq & \liminf_{ R\to \infty}    P*F_{R}(u/2) 
					+ \int_{{\Lambda}}  \liminf_{ R\to \infty} \bigg\{- \sum_{\stackrel{n \neq 0}{|n_j|\leq R} } \prod_{j=1}^{d}
					\bigg( 1-\frac{|n_{i}|}{R+1} \bigg)    \ft{G}_\lambda({an})  (-1)^{u\cdot n}\bigg\}d\nu(\lambda)\nonumber \\ 
					&=& \liminf_{ R\to \infty}    P*F_{R}(u/2)
					+ \int_{\Lambda}  -\sum_{n \neq 0}  \ft{G}_\lambda({an}) (-1)^{u\cdot n} d\nu({\lambda})  \label{eq-fatou2}.
				\end{eqnarray}
Using the properties of the Fourier transform of $G_\lambda(x)$ and the definition of the Theta Function (\ref{theta-function}), we find that the term inside the integral in (\ref{eq-fatou2}) is equal to
$$
\prod_{j=1}^{d}\lambda_{j}^{-\hf}-\prod_{j=1}^{d}\Theta(\hf;ia_j^2/\lambda_j)\lambda_{j}^{-\hf}.
$$
This proves the lower bound estimate (\ref{G-min-bound}), since  $P*F_R$ is a non-negative function.

\noindent Now, we turn to the proof of the $L^1$ convergence (\ref{L1-min-conv}). By conditions (C1),(C2),(C3) and (C4-) we are at the position of applying Theorem \ref{existence} for the functions $L_{\lambda,a}(z)$ defined at the previous section.
Let $\LL_a(z)$ be the minorant given by Theorem \ref{existence} part (ii), for $K=Q(a)$, $G_\lambda(x)=G(x;\lambda)$ and $F(z;\lambda)=L_{\lambda,a}(z)$. We obtain 
\begin{eqnarray}\label{G-L-int}
&&\int_{\R^d}\G_\lambda(x)-\LL_{a}(x)dx\nonumber = \int_\Lambda\int_{\R^d}G_\lambda(x)-L_{\lambda,a}(x)dxd\nu(\lambda) \nonumber =\\ &&\int_\Lambda \bigg\{1- \bigg\{ \sum_{j=1}^d\frac{\Theta(\hf;ia_j^2/\lambda_j)}{\Theta(0;ia_j^2/\lambda_j)}-(d-1)\bigg\} \prod_{j=1}^d\Theta(0;ia_j^2/\lambda_j)\bigg\}\prod_{j=1}^d\lambda_j^{-\hf}\d\nu(\lambda).
\end{eqnarray}\smallskip

\noindent Using Lemma \ref{lemma-tech-2} we see that the functions inside the integral form a decreasing sequence (indexed by $a\in\Lambda$) converging to zero as $a\uparrow\infty$. Therefore, by the monotone convergence theorem, the integral goes to 0 and the proof is complete.
\end{proof}

\begin{proof}[Proof of Lemma \ref{lemma-tech-2}]
To see that the functions
\begin{equation}
\phi_a:\lambda\in\Lambda\mapsto \bigg\{ \sum_{j=1}^d\frac{\Theta(\hf;ia_j^2/\lambda_j)}{\Theta(0;ia_j^2/\lambda_j)}-(d-1)\bigg\} \prod_{j=1}^d\Theta(0;ia_j^2/\lambda_j)
\end{equation}
form an increasing sequence of functions indexed by $a\in\Lambda$ as each $a_j\uparrow \infty$, is enough to prove that for every $\lambda\in\Lambda$ and $k\in\{1,...,d\}$
$$
\frac{\partial}{\partial a_k}\phi_a(\lambda)\geq 0.
$$
For all $t>0$ denote 
$$
u_j(t)=\Theta(0;it^2/\lambda_j) \mbox{\,\,\,\,\, and \,\,\,\,\,} v_j(t)=\Theta(\hf;it^2/\lambda_j),
$$ 
we can use the Leibniz rule to obtain
\begin{equation}\label{form-der-phi}
\frac{\partial}{\partial a_k}\phi_a(\lambda) = \bigg(v'_k(a_k)+\bigg\{ \sum_{\stackrel{j=1}{j\neq k}}^d\frac{v_j(a_j)}{u_j(a_j)}-(d-1)\bigg\}u'_k(a_k)\bigg)\prod_{\stackrel{j=1}{j\neq k}}^du_j(a_j).
\end{equation}

% \frac{u_k(a_k)v'_k(a_k)- u'_k(a_k)v_k(a_k)}{u_k(a_k)^2}\prod_{j=1}^du_j(a_j) + \bigg\{ \sum_{j=1}^d\frac{v_j(a_j)}{u_j(a_j)}-(d-1)\bigg\}\prod_{\stackrel{j=1}{j\neq k}}^du_j(a_j)u'_k(a_k) \\ &

Using the summation formula (\ref{theta-function}), we see that the functions $t\mapsto (u_j(t)-1)$ is a sum of positive decreasing functions that decreases to $0$ as $t\uparrow \infty$, thus $u_j(t)$ is a decreasing function that decreases to $1$ as $t\uparrow \infty$. Analogously, using the product formula (\ref{theta-prod-rep}), each $v_j(t)$ is a product of positive increasing functions that increases to $1$ as $t\uparrow \infty$, thus $v_j(t)$ is a positive increasing function that increases to $1$ as $t\uparrow \infty$.
Therefore, the term inside the parenthesis in (\ref{form-der-phi}) is positive, which implies that $\frac{\partial}{\partial a_k}\phi_a(\lambda)>0$, and this proves item (i). 

\noindent Since
$$
\phi_a(\lambda)=\bigg\{ \sum_{j=1}^d\frac{v_j(a_j)}{u_j(a_j)}-(d-1)\bigg\}\prod_{j=1}^du_j(a_j),
$$ 
we see that $\phi_a(\lambda)$ converges to $1$, for every $\lambda\in\Lambda$, as each $a_j\uparrow\infty$ and this proofs item (ii).

\end{proof}

%------------------------------------------------------------------
%------------------------------------------------------------------
%-------------------------------------------------------------------

\section{Further Results}

In this section we will use the machinery of the previous section to construct one--sided approximations by trigonometric polynomials. After, as in Part III of \cite{CLV}, we will give some examples of functions that our method is applicable and then we present some Hilbert--type inequalities that arise from the constructions of section \S\ref{Gaussian Subordination Method}.

\subsection{Periodic Analogues}\label{per-ana}

In this subsection we find the best approximations by trigonometric polynomials for functions that are, in some sense, subordinate to Theta functions. The proofs of the theorems in this section are almost identical to the proofs of the previous sections, and thus we state the theorems without proof.

\begin{definition}
 Let $a=(a_{1},a_{2},...,a_{d})\in\mathbb Z^{d}_{+}$ ( i.e\ \ $a_{j}\geq 1$  $\forall j$ ) ,  we will say that the degree of a trigonometric Polynomial $P(x)$ is less than $a$ (degree $P<a$ ) if
\begin{equation*}
P(x)=\sum_{-a<n<a}\ft{P}(n)e(n \cdot x).
\end{equation*}
\end{definition}
The problems we are interested to solve have the following general form

\noindent {\bf \small Periodic Majorization Problem.} Fix an $a\in\Z^d_+$  (called the {\it degree}) and a Lebesgue measurable real periodic function $g:\T^d\to \R$. Determine the value of 
\begin{equation}\label{opt-ans}
	\inf \int_{\T^d}|F(x)-g(x)| dx,
\end{equation}
where the infimum is taken over functions $F:\T^d\to\R$  satisfing
\begin{enumerate}\label{maj-cond}
	\item[(i)] $F(x)$ is a real trigonometric polynomial.
	\item[(ii)] Degree of $F(x)$ is less than $a$.
	\item[(iii)] $F(x)\geq g(x)$ for every $x\in\T^d$.
\end{enumerate} 
If the infimum is achieved, then identify the extremal functions $F(z)$.
Similarly there exist the minorant problem
\smallskip

\noindent {\bf \small Periodic Minorization Problem.} Solve the previous problem with condition (iii) replaced by the condition
\begin{equation*}\label{min-cond}
 \mbox{(iv) } F(x)\leq g(x) \, \mbox{ for every } \, x\in\T^d.
\end{equation*}

Now we define for every $\lambda\in\Lambda$, the periodization of the Gaussian function $G_\lambda(x)$ by
\begin{equation*}
f_{\lambda}(x):=\sum_{n\in \Z^{d}}G_{\lambda}(x+n)=\sum_{n\in\mathbb Z^{d}}\prod_{j=1}^{d}e^{-\lambda_{j}\pi(x_{j}+n_{j})^{2}}=\prod_{j=1}^{d}\Theta(x_{j};i/\lambda_{j})\lambda^{-\frac{1}{2}}_{j}.
\end{equation*}

\begin{theorem}[Existence]\label{existence per}
For a given $a=(a_{1},a_{2},...,a_{d})\in\Z^{d}_{+}$, let $\Lambda$ be a measurable space of parameters,  and for each $\lambda\in\Lambda, $ let $R_{\lambda}(x)$ be a real trigonometric polynomial with degree less than $a$. Let $\nu$ be a non--negative measure in $\Lambda$ that satisfies
\begin{equation}\label{cond exist G.S}
\int_{\Lambda}\int_{\T^d}| R_{\lambda}(x)-f_{\lambda}(x)|dxd\nu(\lambda)<\infty.
\end{equation}
Suppose that $g:\T^d\rightarrow \R$ is a continuous periodic function such that
\begin{equation}\label{cond exist G.S 2}
\ft{g}(k)=\int_{\Lambda}\ft{G}_{\lambda}(k)d\nu(\lambda), 
\end{equation}
for all $k\in\Z^{d}$ such that $|k_{j}|\geq a_{j}$ for some $j\in\{1,2,...,d\}$. Then
\begin{enumerate}
\item[(i)] if $f_{\lambda}(x)\leq R_{\lambda}(x)$ for each $x\in\T^{d}$ and $\lambda\in\Lambda,$ then there exist a trigonometric polynomial $m_a(x)$ with degree $m_a < a$, such that of $m_a(x)\geq g(x)$ for all $x\in\T^{d}$ and 
\begin{equation*}
\int_{\T^{d}}m_a(x)-g(x)dx
\end{equation*}  
is equal to the (LHS) of (\ref{cond exist G.S}).
\item[(ii)] if $R_{\lambda}(x)\leq f_{\lambda}(x)$ for each $x\in\T^{d}$ and $\lambda\in\Lambda,$ then there exist a trigonometric polynomial $l_a(x)$ with degree $l_a < a$, such that $l_a(x)\leq g(x)$ in $\T^{d},$ and
\begin{equation*}
\int_{\T^{d}}g(x)-l_a(x)dx,
\end{equation*}  
is equal to (LHS) of (\ref{cond exist G.S}).
\end{enumerate}
\end{theorem}

Before we state the main theorems of this section we need some definitions. The functions $M_{\lambda,a}(x)$ and $L_{\lambda,a}(x)$, defined in (\ref{maj-def}) and (\ref{min-def}), belong to $L^1(\R^d)$,  thus, by the Plancharel-P\'olya theorem (see \cite{PP}) and the periodic Fourier inversion formula, their respective periodizations are trigonometric polynomials of degree less than $a$, that is
\begin{equation}\label{per-maj}
m_{\lambda,a}(x):=\sum_{n\in \Z^{d}}M_{\lambda,a}(x+n)=\sum_{-a<n<a}\ft{M}_{\lambda,a}(n)e^{2\pi in \cdot x}
\end{equation}
and
\begin{equation}\label{per-min}
l_{\lambda,a}(x):=\sum_{n\in \Z^{d}}L_{\lambda,a}(x+n)=\sum_{-a<n<a}\ft{L}_{\lambda,a}(n)e^{2\pi in \cdot x}
\end{equation}
holds for each $x\in\T^d$.
The following theorem offers a resolution to the Majorization Problem for a specific class of functions.

\begin{theorem}[Gaussian Subordination -- Periodic Majorant]\label{G.S.P.M}
Let $a\in\Z^{d}_{+}$ and $\nu$ be non-negative Borel measure on $\Lambda$ that satisfies 
\begin{equation}\label{per maj nu cond}
						\int_{\Lambda} \prod_{j=1}^{d}\lambda_{k}^{-\hf}\bigg\{ \prod_{j=1}^{d}\Theta(0;ia_j^2/\lambda_j) - 1\bigg\}    \d\nu({\lambda}) <\infty.
					\end{equation}
					\smallskip
					
\noindent Let $g:\T^d\rightarrow \R$ be a continuous periodic function such that
\begin{equation}\label{cond-exist-g-maj}
\ft{g}(k)=\int_{\Lambda}\ft{G}_{\lambda}(k)d\nu(\lambda), 
\end{equation}
for all $k\in\Z^{d}$ such that $|k_{j}|\geq a_{j}$ for some $j\in\{1,2,...,d\}$. Then for every real trigonometric polynomial $P(x)$, with degree $P <a$ and $P(x)\geq g(x)$ for all $x\in\T^d$, we have

\begin{equation} \label{opt-per-maj}
\int_{\T^d}P(x)-g(x)dx \geq \int_{\Lambda} \prod_{j=1}^{d}\lambda_{k}^{-\hf}\bigg\{ \prod_{j=1}^{d}\Theta(0;ia_j^2/\lambda_j) - 1\bigg\}   \d\nu({\lambda}).
\end{equation} 
\smallskip

\noindent Moreover, there exists a real trigonometric polynomial $m_a$, with degree $m_a<a$, such that $m_a(x)$ is a majorant of $g(x)$ that interpolates $g(x)$ on the lattice $\Z^d/a$ and equality at (\ref{opt-per-maj}) holds.

\end{theorem}

%-----------Gaussian subordination Periodic Majorant------------------------------------
%----------------------------------------------------------------

%------------------------G.S Periodic Minorant!-----------------
%---------------------------------------------------------------
%----------------------------------------------------------------

\begin{theorem}[Gaussian Subordination -- Periodic Minorant]\label{gauss-sub-min-per}
Let $a\in\Z^d_+$ and $\nu$ be non-negative Borel measure on $\Lambda$ such that

 					\begin{equation}\label{min-ext-cond-per}
 					 \int_\Lambda \bigg\{1- \bigg\{ \sum_{j=1}^d\frac{\Theta(\hf;ia_j^2/\lambda_j)}{\Theta(0;ia_j^2/\lambda_j)}-(d-1)\bigg\} \prod_{j=1}^d\Theta(0;ia_j^2/\lambda_j)\bigg\}\prod_{j=1}^d\lambda_j^{-\hf}\d\nu(\lambda)
<\infty.
 					\end{equation}
 					\smallskip
 					
\noindent				Let $g:\T^{d} \to \R$ be continuous periodic function such that
					\begin{equation}\label{g-Fourier2}
							\ft{g}(k)=\int_{\Lambda}\ft{G}_{\lambda}({ k})d\nu(\lambda) 
					\end{equation}
				for all $k\in\Z^{d}$ such that $|k_{j}|\geq a_{j}$ for some $j\in\{1,2,...,d\}$. Then, if $P(z)$ is a real trigonometric polynomial with degree less than $a$ that minorizes $g(x)$, we have
				
				\begin{equation}\label{g-min-bound}
				\int_{\T^d}  g({ x})-P(x)  d{ x}\geq \int_{\Lambda} \prod_{j=1}^{d}\lambda_{k}^{-\hf}\bigg\{ 1-\prod_{j=1}^{d}\Theta(\hf;ia_j^2/\lambda_j)\bigg\}    d\nu({\lambda}).				
				\end{equation}\smallskip
				
\noindent Furthermore, there exists a family of trigonometric polynomial minorants $ \{l_a({ x}):a\in\Z^d_+\}$ with degree $l_a<a$, such that the integral
$$
\int_{\T^d}  g(x)-l_{a}({ x})   d{ x}
$$
is equal to the quantity in (\ref{min-ext-cond-per}), and
\begin{equation}\label{L1-min-conv-per}
\lim_{a\uparrow\infty}\int_{\T^d}  g(x)-l_{a}({ x})   d{ x}=0,
\end{equation}
					where $a\uparrow\infty$ means ${a_j}\uparrow \infty$ for every $j$.
\end{theorem}
\smallskip
 
\begin{corollary}
Under the hypothesis of Theorem \ref{gauss-sub-min-per}, suppose also that there exist an $R>0$ such that $\supp(\nu)\subset \Lambda\cap Q(R)$, and
\begin{equation*}
\int_{\T^d}g(x)dx=\int_{\Lambda}\prod_{j=1}^d\lambda_j^{-\hf}\d\nu(\lambda)<\infty.
\end{equation*}
Then, there exist a constant $\alpha_0>0$, such that if $\alpha:=\min\{a_{j}\}\geq \alpha_0$ and if $P(x)$ is a trigonometric polynomial with degree $P<a$ that minorizes $g(x)$, we have
\begin{equation*}
\int_{\T^d}P(x)dx\leq(1+5de^{-\alpha^{2}/R})\int_{\T^d}l_{a}(x)dx.
\end{equation*}

\end{corollary}
\medskip

%\noindent The proof of this corollary is completely analogous to the proof of the Corollary \ref{cor1} in section \S\ref{M.G.F} and we omit it.
%\smallskip 			

\subsection{The Class of Contemplated Functions}\label{classes}

We define the class 
$$
\GG^d=\bigcap_{a\in\Lambda}\GG_-^d(a)\cap\GG_+^d(a).
$$
This is the class of pairs such that Theorems \ref{gauss-sub-maj} and \ref{gauss-sub-min} are applicable for every $a\in\Lambda$.
In this subsection we present conditions for a pair $(\G,\nu)$ belong to this class.

Some interesting properties arise when $\nu$ is concentrated in the diagonal. For every $\eta\in [0,1]$ we define $\Lambda_\eta =\{\lambda\in\Lambda : \eta\lambda_j\leq\lambda_k \forall j,k \}$ and we note that $\Lambda_0=\Lambda$ and $\Lambda_1=\{\lambda\in\Lambda:\lambda_j=\lambda_k \forall j,k \}$ is the diagonal.

\begin{proposition}\label{prop-nu-diag}
Let $(\G,\nu)$ be a pair that satisfies conditions {\bf (C{1})},{\bf (C2)} and {\bf (C3)} for every $a\in\Lambda$. Suppose that $\supp(\nu)\subset \Lambda_\eta$ for some $\eta\in(0,1]$ and $\nu(\Lambda_\eta\setminus Q(R))<\infty$ for every $R>0$. Then $(\G,\nu)\in\GG^d$.

\end{proposition}

\begin{proof}
We only prove that condition {\bf (C4-)} holds, the condition {\bf (C4+)} is analogous. 
Given an $a\in\Lambda$, define the function

\begin{equation*}
\phi_a:\lambda\in\Lambda\mapsto \bigg( \sum_{j=1}^d\frac{\Theta(\hf;ia_j^2/\lambda_j)}{\Theta(0;ia_j^2/\lambda_j)}-(d-1)\bigg) \prod_{j=1}^d\Theta(0;ia_j^2/\lambda_j).
\end{equation*}
\smallskip

\noindent By (\ref{theta-function}) and the Poisson summation formula, we have the following estimates
\begin{eqnarray}\label{theta-hf-est}
1-\Theta(\hf;i/t) \sim 2e^{-\pi/t} \,\, \mbox{ as } t\to 0, \\
\Theta(\hf;i/t) \sim 2t^{\hf}e^{-\pi t/4} \,\, \mbox{ as } t\to \infty
\end{eqnarray}
and 
\begin{eqnarray}
\Theta(0;i/t)-1 \sim 2e^{-\pi/t} \,\, \mbox{ as } t\to 0, \\ \label{theta-0-est}
\Theta(0;i/t) \sim t^{\hf} \,\, \mbox{ as } t\to \infty
\end{eqnarray}
where the symbol $\sim$ means that the quotient converges to $1$. Using the (LHS) inequality of Lemma \ref{theta-lemma} we conclude that exists an $R>0$ and a $C>0$ such that 
$$
\phi_a(\lambda)\geq 1-C\sum_{j=1}^de^{-\pi a_j^2/\lambda_j}
$$
for every $\lambda\in\Lambda\cap Q(R)$. Choose $l\in\{1,...,d\}$ such that $a_l\leq a_j$ for every $j$. If $\lambda\in\Lambda_\eta\cap Q(R)$ we have

\begin{eqnarray*}
\big\{1-\phi_a(\lambda)\big\}\prod_{j=1}^d\lambda_j^{-\hf} \leq C\bigg(\sum_{j=1}^de^{-\pi a_j^2/\lambda_j}
\bigg)\prod_{j=1}^d\lambda_j^{-\hf} \leq dCe^{-\pi a_l^2\eta/\lambda_l} \prod_{j=1}^d\lambda_j^{-\hf} \\ \leq dC \prod_{j=1}^d e^{-\pi a_l^2\eta^2/(d\lambda_j)} \prod_{j=1}^d\lambda_j^{-\hf} = dC\ft{G}_\lambda(\beta u),
\end{eqnarray*}
where $\beta=a_l\eta d^{-\hf}$ and $u=(1,...,1)$.
\noindent By estimates (\ref{theta-hf-est})-(\ref{theta-0-est}) we see that the functions $\Theta(\hf;i/t)t^{-\hf}$ and $\Theta(0;i/t)t^{-\hf}$ are bounded for $t\in[\eta R,\infty)$, and thus, we conclude that the function 
$$
\lambda\in\Lambda\mapsto\phi_a(\lambda)\prod_{j=1}^d\lambda_j^{-\hf}
$$
is bounded on $\Lambda_\eta\setminus Q(R)$, since it is a finite sum of products of these theta functions. 

\noindent Since $\eta>0$, we obtain that the function
$$
\lambda\in\Lambda\mapsto\big\{1-\phi_a(\lambda)\big\}\prod_{j=1}^d\lambda_j^{-\hf}
$$
is bounded in $\Lambda_\eta\setminus Q(R)$, let say by $C'$.
Therefore, we have

\begin{eqnarray*}
\int_{\Lambda_\eta} \big\{1-\phi_a(\lambda)\big\}\prod_{j=1}^d\lambda_j^{-\hf}\d\nu(\lambda) \leq dC \int_{\Lambda_\eta\cap Q(b)}\ft{\G}_\lambda(\beta u)\d\nu(\lambda) + C'\nu(\Lambda_\eta\setminus Q(b))<\infty,
\end{eqnarray*}\smallskip

\noindent which is finite by condition {\bf (C3)} and the hypotheses of this lemma.
Thus $\nu$ satisfies condition {\bf (C4-)} and this concludes the proof.

\end{proof}
\smallskip

\begin{proposition}\label{prop-nu-prob}
Let $\nu$ be a probability measure on $\Lambda$ with $\supp(\nu)\subset \Lambda_\eta$ for some $\eta\in(0,1]$. Define the function
\begin{equation}\label{g-nu-prob-def}
\G(x)=\int_{\Lambda}G_\lambda(x)\d\nu(\lambda)
\end{equation}
for all $x\in\R^d$. Then $(\G,\nu)\in\GG^d$.
\end{proposition}

\begin{proof}
Is easy to see that $\G$ is a bounded, continuous and radially decreasing function, thus $\G$ satisfies conditions {\bf (C1)}. Note that for every $x\neq 0$, the function
$$
\lambda\in\Lambda \mapsto \ft{G}_\lambda(x)
$$
is bounded on $\Lambda_\eta$ and this bound can be taken uniform for $x$ outside any neighborhood of the origin. Thus condition {\bf (C2)} holds, and using Fubini's theorem, condition {\bf (C3)} also holds. We conclude, by Proposition \ref{prop-nu-diag}, that $(\G,\nu)\in\GG^d$.

\end{proof}
\smallskip

\noindent Due to a clasical result of Schoenberg (see \cite{Sch}), a radial function $\G(x)=\G(|x|)$ admits the representation (\ref{g-nu-prob-def}) for a probability $\nu$ supported on the diagonal $\Lambda_1$ if and only if the radial extension to $\R^{n}$ of $\G(r)$ is positive definite, for all $n>0$. And this occurs if and only if the function $\G(r^{\hf})$ is completely  monotone. As consequence of this fact and Proposition \ref{prop-nu-prob} the following multidimensional versions of the functions in Section 11 of \cite{CLV} are contemplated

\begin{example}
\begin{equation*}
g(x)=e^{{-\alpha}|x|^{r}}\in\GG^d,\ \ \alpha>0 \ \ and\ \ 0<r<2.
\end{equation*}
\end{example}

\begin{example}
\begin{equation*}
g(x)=(|x|^{2}+\alpha^{2})^{-\beta}\in\GG^d,\ \alpha>0\ and\ \beta>0.
\end{equation*}
\end{example}

\begin{example}
\begin{equation*}
g(x)=-\log\left(\frac{|x|^{2}+\alpha^{2}}{|x|^{2}+\beta^{2}}\right)\in\GG^d,\ for \ 0<\alpha<\beta.
\end{equation*}

\end{example}
\noindent The following example is a high dimensional analogue of Corollary 21 \cite{CLV}. 
\begin{example}\label{G-mod-pow} Given a $\sigma\in(0,\infty)\setminus 2\Z_+$ consider the measure $\nu_{\sigma}$, that is supported on the diagonal $\Lambda_1$ and defined on a Borel set $E\subset\Lambda$ by 

\begin{equation}\label{sigma-meas}
\nu_\sigma(E)=C(\sigma)\int_{P(E\cap \Lambda_1)}t^{-\tfrac{\sigma}{2}-1}\dt
\end{equation}\smallskip

\noindent where $P:\Lambda_1\to\R$ is the projection $P(tu)=t$ for all $t\geq 0$ and 
$$
C(\sigma)=\pi^{-\frac{\sigma}{2}}\big/\Gamma(-\frac{\sigma}{2}).
$$
where $\Gamma(z)$ is the classical Gamma function. Define on $\R^d$ the function

\begin{equation*}
\G_\sigma(x)=|x|^{\sigma} \, \text{ for }\ \sigma\in(0,\infty)\setminus 2\Z_+
\end{equation*}
\smallskip

\noindent we claim that $(\G_\sigma,\nu_\sigma)\in\GG^d$. To see this, first we note that Lemma 18 of \cite{CLV} has a trivial generalization to the several variable setting, that is, for every function $\varphi\in \mc S(\R^d)$ that is zero on a neighborhood of the origin, we have

$$
\int_{\R^d}\ft{\varphi}(x)|x|^\sigma\dx=A(d,\sigma)\int_{\R^d}\varphi(x)|x|^{-d-\sigma}\dx,
$$
where
\smallskip

$$
A(d,\sigma)= \pi^{-\sigma-d/2}\Gamma(\frac{d+\sigma}{2})/\Gamma(-\frac{\sigma}{2}).
$$
Secondly, note that

$$
\frac{A(d,\sigma)}{|x|^{d+\sigma}}=\int_\Lambda \ft{G}_\lambda(x)\d\nu_\sigma(\lambda).
$$
\smallskip

\noindent Hence, by Proposition \ref{prop-nu-diag} the pair $(\G_\sigma,\nu_\sigma)$ belongs to  $\GG^d$.
\end{example}

\subsubsection{The Periodic Case}

Given a pair $(\G,\nu)\in\GG^d$, suppose that $\G\in L^1(\R^d)$ and the periodization 
$$
x\in \T^d \mapsto \sum_{n\in\Z}\G(n+x)
$$ 
is equal almost everywhere to a continuous function $g(x)$. We easily see that the pair $(g,\nu)$ is contemplated by the Theorems \ref{G.S.P.M} and \ref{gauss-sub-min-per} for every degree $a\in\Z^d_+$. Thus, the period method contemplates the following functions

\begin{example}
\begin{equation}
g(x)=\prod_{j=1}^d\Theta(x_j;i/\lambda_j)=\prod_{j=1}^d\lambda_j^{\hf}\sum_{{n\in\Z^{d}}}G_\lambda(n+x) ,\,\, \mbox{for all } \lambda\in\Lambda.
\end{equation}
\end{example}
\smallskip

\begin{example}
\begin{equation}
g(x)=\sum_{{n\in\Z^{d}}}e^{-\alpha|n+x|^r} ,\,\, \mbox{for all } \alpha>0 \mbox{ and } 0<r<2.
\end{equation}
\end{example}
\smallskip

However, we cannot use this construction for the case of the functions $\G_\sigma(x)=|x|^\sigma$ of example $\ref{G-mod-pow}$. The next proposition tell us that if the Fourier coefficients of $\G$  decay sufficiently fast, then the periodization of $\G$ via Poisson summation formula is contemplated by the periodic method.

\begin{proposition}
Let $(\G,\nu)\in\GG^d$. Suppose that exist constants $C>0$ and $\delta>d/2$ such that
$$
\int_{\Lambda}\ft{G}_\lambda(x)\d\nu(\lambda) \leq C|x|^{-\delta}
$$
if $|x|\geq 1$. Define the function
$$
x\mapsto \displaystyle\lim_{N\ra \infty}\sum_{\stackrel{n\in\Z^{d}}{0<|n|<N}} \int_{\Lambda}\ft{G}_\lambda(n)\d\nu(\lambda) e(n\cdot x),
$$
where the limit is taken in $L^2(\T^d)$, and suppose this can be identified with a continuous function $g(x)$. Then the pair $(g,\nu)$ satisfies all the conditions of Theorems \ref{G.S.P.M} and \ref{gauss-sub-min-per} for every $a\in\Z^d_+$.
\end{proposition}

With this last proposition we see that the following example is contemplated by the periodic method of subsection \ref{per-ana}.

\begin{example}
\begin{equation}
g_\sigma(x)=\sum_{\stackrel{n\in\Z^{d}}{n\neq 0}}|n|^{-d-\sigma}e(n\cdot x) ,\,\, \mbox{for all } \sigma>0.
\end{equation}
\end{example}
\smallskip

\subsection{Hilbert-Type Inequalities}\label{Hilbert}

%In this section we start presenting a classical result about Hilbert-Type inequalities using the pairs $(\G,\nu)\in\GG^d$ (see \cite{CLV} corollary 22). 
Given an $a\in\Lambda$ and $x\in\R^d$ define the norm
$$
|x|_a=\max\{|x_j/a_j|: j=1,...,d\}.
$$
Given an $a\in\Lambda$ we say that a sequence $\{\xi_k\}_{k\in\Z}$ of vectors in $\R^d$ is $a$-separated if $|\xi_k-\xi_l|_a\geq 1$ for all $l\neq k$. We have the following proposition

\begin{proposition}\label{prop-hilb-type}
Let $(\G,\nu)\in\GG^d$, $a\in\Lambda$ and $\{\xi_k\}_{k\in\Z}$ an $a$-separated sequence of vectors. Then for every finite sequence of complex numbers $\{w_{-N},...,w_0,...,w_N\}$, we have
\begin{equation}\label{hilb-ineq}
-A(a,d,\nu)\sum_{n=-N}^N|w_n|^2 \leq \sum_{\stackrel{n,m=-N}{n\neq m}}^Nw_n\overline{w}_m\ft{\G}(\xi_n-\xi_m) \leq B(a,d,\nu)\sum_{n=-N}^N|w_n|^2
\end{equation}
where $A(a,d,\nu)$ is equal to the quantity (\ref{min-nu-cond}) and $B(a,d,\nu)$ is equal to the quantity (\ref{maj-nu-cond}). Furthermore the constant $B(a,d,\nu)$ is sharp.
\end{proposition}

\noindent {\it Remark}: For $\xi \neq 0$ we write
$$
\ft{\G}(\xi)= \int_{\Lambda}\ft{G}_\lambda(\xi)\d\nu(\lambda).
$$

\begin{proof}
To prove inequality in (\ref{hilb-ineq}), define the function $D=\M_a-\G$ where $\M_a$ is given by Theorem \ref{gauss-sub-maj}. Since $\M_a$ if of exponential type with respect to $Q(a)$ we obtain
$$
\sum_{{n,m=-N}}^Nw_n\overline{w}_m\ft{\D}(\xi_n-\xi_m) = B(a,d,\nu)\sum_{n=-N}^N|w_n|^2- \sum_{\stackrel{n,m=-N}{n\neq m}}^Nw_n\overline{w}_m\ft{\G}(\xi_n-\xi_m).
$$
But the (LHS) of this last equality is positive since $\D(x)$ is a non-negative function. In an analogous way we can prove the (LHS) inequality of (\ref{hilb-ineq}).

\noindent For the sharpness of the (RHS) inequality in (\ref{hilb-ineq}) suppose that we could change $B(a,d,\nu)$ by some other constant $B'$. Consider the sequence $w_n=1$ for all $n\in\Z$ and let the sequence $\{\xi_n\}$ be an enumeration of the points in the set $J(R)=Q(Ra)\cap a\Z^d$ where $R\in\Z_+$. We obtain
$$
-\sum_{\xi,\xi'\in J(R)}\ft{\D}(\xi-\xi') =\sum_{\stackrel{\xi,\xi'\in J(R)}{\xi\neq \xi'}}\ft{\G}(\xi-\xi') - (2R+1)^d\ft{\D}(0) \leq (2R+1)^d(B'-\ft{D}(0)).
$$
On the other hand we have
$$
\sum_{\xi,\xi'\in J(R)}\ft{\D}(\xi-\xi') =\sum_{|n_j|\leq 2R} \bigg\{\prod_{j=1}^d(2R+1-|n_j|)\bigg\}\ft{\D}(an). 
$$
Hence, we conclude that
$$
B'-\ft{\D}(0)\geq -\sum_{|n_j|\leq 2R} \prod_{j=1}^d\bigg(1-\frac{|n_j|}{2R+1}\bigg)\ft{\D}(an)-=F_{2R}*P(0),
$$
where $F_{2R}(x)$ is the F\'ejer's Kernel defined in (\ref{fejer-ker}) and 
$$
P(x)=\frac{1}{\alpha}\sum_{n\in\Z^d}\D((n+x)/a),
$$
with $\alpha=a_1...a_d$. Replying the arguments of the Theorem \ref{gauss-sub-maj} we would conclude that 
$$
\liminf_{R\to\infty}F_{2R}*P(0)=0
$$
and this implies that $B'\geq \ft{\D}(0)$, since $\ft{\D}(0)=B(a,d,\nu)$ this concludes the proof.

\end{proof}

\noindent The next corollary is a generalization of Corollary 22 of \cite{CLV} in the multidimensional setting and is a direct application of Proposition \ref{prop-hilb-type} for example \ref{G-mod-pow}. Below
$|\cdot|$ stands for the Euclidean norm in $\R^d$. 
\begin{corollary}

Let $\sigma>0$, $a\in\Lambda$ and $\{\xi_k\}_{k\in\Z}$ an $a$-separated sequence of vectors. Then for every finite sequence of complex numbers $\{w_{-N},...,w_0,...,w_N\}$, we have
\begin{equation}
-A(a,d,\sigma)\sum_{n=-N}^N|w_n|^2 \leq \sum_{\stackrel{n,m=-N}{n\neq m}}^N\frac{w_n\overline{w}_m}{|\xi_n-\xi_m|^{d+\sigma}} \leq B(a,d,\sigma)\sum_{n=-N}^N|w_n|^2,
\end{equation}
where 
$$
A(a,d,\sigma)=-\sum_{j=1}^d\,\sum_{n\in\Z^d\setminus{0}}\frac{(-1)^{n_j}}{|an|^{d+\sigma}} \,+\, (d-1)\sum_{n\in\Z^d\setminus{0}}\frac{1}{|an|^{d+\sigma}}
$$ 
and 
$$
B(a,d,\sigma)=\sum_{n\in\Z^d\setminus{0}}\frac{1}{|an|^{d+\sigma}}.
$$
Furthermore, the constant $B(a,d,\sigma)$ is sharp.

\end{corollary}

\section{Concluding Remarks} \label{conclusion}
We mentioned in the introduction that Selberg generalized his construction of majorizing and minorizing the characteristic function of an interval, to majorizing and minorizing the characteristic function of a box by functions whose Fourier transforms are supported in a (possibly different) box. Another way to generalize Selberg's original construction to the several variables setting is to consider majorizing and minorizing the characteristic function of a ball by functions whose Fourier transforms are supported in a (possibly different) ball. This problem was considered by Holt and Vaaler \cite{HV} and their methods were recently extended by Carneiro and Littmann \cite{CL3}. 

As far as we know, almost nothing is known about the Beurling-Selberg extremal problem in higher dimensions when the Fourier transform is supported on a fixed symmetric convex body $K$. The following question (perhaps the simplest Beurling-Selberg extremal problem in higher dimensions) is open:
\begin{quote}
	Let $K$ be a symmetric convex body in $\R^d$. Determine the value of 
		\[
			\eta(K)=\inf\dint_{K}F(x)dx
		\]
	where the infimum is taken over continuous integrable functions $F:\R^d\ra \R$ that satisfy (i) $F(0)\geq1$, (ii) $F(x)\geq0$ for all $x\in\R^d$, and (iii) $\hat{F}(\xi)=0$ if $\xi\not\in K$.
\end{quote}
In one dimension, the solution to this problem is the Fejer kernel for $\R$. So the solution to the above problem can be thought of as an analogue of the Fejer kernel for $K$.\newline
\indent  It is conjectured \cite{BK} by the third named author and Jeffrey Vaaler, that $\eta(K)=2^{d}/\mathrm{vol}_{d}(K)$.  Vaaler \cite{VPC} has shown that this conjecture is true if $K$ is an {\it extremal body}. That is, if $K$ achieves equality in Minkowski's convex body theorem, then the conjecture in \cite{BK} is true. For example, the regular hexagon in $\R^2$ is an extremal body. The only other body $K$ for which the conjecture is known to hold  is the Euclidean ball.

%%%%%%%%%%%%%%%%%%%%%%%%%%%%%%%%%%%%%%%%%%%%%%%%%%%%%%%%%%%%%%%%%%

\nocite{*}
\bibliographystyle{plain}
\bibliography{gaussian}	

\end{document}